\documentclass[amscd,amssymb,verbatim,12pt]{amsart}
\theoremstyle{plain}
\newtheorem{Thm}{Theorem}
\newtheorem{Cor}{Corollary}

\newtheorem{Lem}{Lemma}
\newtheorem{Prop}{Proposition}
\theoremstyle{definition}
\newtheorem{Def}{Definition}
\theoremstyle{remark}

\newcommand {\n}{\noindent}
\newcommand \vsp {\vspace*{10pt}}

\newcommand{\R}{{\mathbb R}}
\newcommand{\N}{\mathbb N}

\newcommand{\Z}{\mathbb Z}

\newcommand{\Dli}{{\mathcal{D}}'}

\newcommand{\vspp}{\vspace*{5pt}}

\newcommand{\CB}{{\mathcal{B}}}

\newcommand{\bM}{\mathbf{M}}
\newcommand{\bN}{\mathbf{N}}
\newcommand{\Rou}[1]{{\mathcal E}^{\{#1\}}}
\newcommand{\E}{\mathcal{E}}
\newcommand{\fa}{\forall}
\newcommand{\ex}{\,\exists\,}
\newcommand{\alp}{\lvert\alpha\rvert}
\newcommand{\trb}{|||}
\usepackage{amsmath, amssymb, colonequals}
\usepackage[usenames]{xcolor}

\usepackage{ulem}
	
\begin{document}

	\title{The Metivier inequality and ultradifferentiable hypoellipticity}
	\author{Paulo D. Cordaro}
	\author{Stefan F\"urd\"os}
		
	\date{July 24. 2023. Corresponding author: stefan.fuerdoes@univie.ac.at}


\maketitle
	
\begin{abstract}
	In 1980 M{\'e}tivier characterized the analytic (and Gevrey) hypoellipticity of $L^2$-solvable
	partial linear differential operators by a-priori estimates.
	  In this note we extend this characterization to  
	 ultradifferentiable hypoellipticity with respect to Denjoy-Carleman classes given by suitable weight sequences.
	 We also discuss the case when the solutions can be taken as hyperfunctions and  present some applications.
\end{abstract}

\section*{Introduction}

In this work we study the regularity of solutions of linear partial differential equations within the framework of Denjoy-Carleman classes defined by appropriate
weight sequences. In fact, our principal concern here is to what extend a result due
to G.~M\'etivier (cf.~\cite{zbMATH03712425}), proved in the study of
analytic and Gevrey regularity, could still be valid in this more general set up.  

More precisely, in \cite{zbMATH03712425}  a characterization of analytic hypoellipticity is presented for $L^2$-solvable linear partial differential equations in terms of a very precise a priori inequality. The author also mentions that a similar characterization for Gevrey hypoellipticity is also valid, and in \cite{BoveMT} the result is extended for pseudodifferential operators.

In the present work we are able to extend this M{\'e}tivier result for what we call here {\sl admissible weight sequences} (see Definition  \ref{AdmissibleDef} below). 
The corresponding Denjoy-Carleman classes contain the Gevrey classes of order $s\geq 1$ properly.
 It is important to note that it is irrelevant in our presentation if the classes are either quasi-analytic or non quasi-analytic.
 
 One of the main points in M{\'e}tivier's argument is to give a positive answer to a question related to the concept of solvability: assume that $P=P(x,D)$ is a real analytic, linear partial
 differential operator in an open set $\Omega\subset\R^n$, which is $L^2$-solvable and Gevrey hypoelliptic of some order $s\geq 1$. Take $f\in {\mathcal G}^s(\Omega)$,
 an open set $U\Subset\Omega$ and let $u\in L^2(U)$ solve $Pu=f$ in $U$ with minimum $L^2$-norm. Automatically $u\in \mathcal{G}^s(U)$. Is it possible to bound the Sobolev norms of $u|_V$, where $V\Subset U$ is another open set, in terms of the Gevrey norms of $f|_U$?
 
 This is what is achieved by M{\'e}tivier when $s=1$.  The extension to the Gevrey case of arbitrary order is straightforward. Indeed, M{\'e}tivier's argument in the analytic case is based on an interpolation method for which it is needed
	  the choice of  suitable subsequences of $(j!)^{1/j}\sim j$ (in this case the subsequences are explicitly described).
         The main difference in the Gevrey case is that now it is needed subsequences of $(j!^s)^{1/j}\sim j^s$, 
	which can be obtained from the ones of the real analytic case after applying the uniform deformation $j\mapsto j^s$, cf.~\cite{BoveMT}.
	
	This deformation argument is no longer possible in the
 Denjoy-Carleman case. 
 The situation is now much more delicate, and the determination of the class of weight sequences for which the
 result is valid is indeed one of the key points in our work (see Definition \ref{AdmissibleDef} below). 
 Moreover the proof of our main result requires several new insights which we believe justifies its publication.

After we discuss several results on weight sequences and Denjoy-Carleman classes is Section 1, we state our main result in Section 2
(Theorem \ref{MainThm}) and prove some consequences of it. Section 3 is devoted to the proof of Theorem 1. Finally, in Section 4, we extend our result to
the hyperfunction set up and discuss the case of Hörmander's sum of squares operators.

	\section{Preliminaries on weight sequences and the corresponding Denjoy-Carleman classes}
		
	We say that	 a sequence of positive numbers $\bM=(M_k)_k$ is a weight sequence
	 if $M_0=1$, $M_1\geq 1$ and $\bM$ is logarithmic convex, i.e.
		\begin{gather}\label{LogConvexity}
			M_k^2\leq M_{k-1}M_{k+1},\quad \fa k\in\N,\\
			\intertext{and}
			\lim_{k\rightarrow\infty}\sqrt[\leftroot{3}\uproot{1}k]{M_k}=\infty.\label{RootsDivergence}
		\end{gather}
	If $\bM$ is a weight sequence the Denjoy-Carleman class 
	${\mathcal E}^{\{\bM\}}(\Omega)$, $\Omega\subseteq\R^n$ open, of ultradifferentiable functions associated
	to $\bM$, consists of all functions $f\in\E(\Omega)$, for which the following holds:
	for every compact set $K\subseteq\Omega$ there are constants $C,h>0$ such that
	\begin{equation*}
		\sup_{x\in K}\,|D^\alpha f(x)|\leq Ch^{\alp} M_{\alp}
	\end{equation*}
for all $\alpha\in\Z_+^n$.
It follows from \eqref{LogConvexity} that ${\mathcal E}^{\{\bM\}}(\Omega)$ is an algebra with respect
to the pointwise addition and multiplication of functions.
	If $\bM=(k!)_k$ then ${\mathcal E}^{\{\bM\}}(\Omega)={\mathcal A}(\Omega)$ is the space of analytic functions on $\Omega$,
	More generally, the Gevrey classes ${\mathcal G}^s(\Omega)$, $s\geq 1,$ are the Denjoy-Carleman classes
	associated to the weight sequences ${\mathbf{G}}^s=(k!^s)_k$.

In order to be able to impose further properties on the classes, we need to discuss some additional 
conditions on the weight sequences we shall deal with.
On the set of weight sequences we can establish the following relation:
if $\bM$ and ${\mathbf N}$ are two weight sequences we set
\begin{align*}
	\bM\preceq {\mathbf N} \quad &:\Longleftrightarrow\quad \exists\, C,h>0:\;\, M_k\leq Ch^k N_k \quad
	\forall k\in\Z_+.\end{align*}
The relation $\preceq$ is both reflexive and transitive.
When we also consider the equivalence relation $\approx$ given by
\begin{align*}
	\bM\approx{\mathbf N} \quad&:\Longleftrightarrow\quad \bM\preceq\bN \;\;\wedge\;\; \bN\preceq\bM
\end{align*}
and identify any pair $\bM,{\mathbf N}$ with $\bM\approx\bN$ then $\preceq$ is also antisymmetric,
i.e.\ $\preceq$ is a partial order. We may write $\bM\precnapprox\bN$ if $\bM\preceq\bN$ and
$\bN\npreceq\bM$.

It is easy to see that ${\mathcal E}^{\{\bM\}}(\Omega)\subseteq{\mathcal E}^{\{\bN\}}(\Omega)$ if
$\bM\preceq{\mathbf N}$.
 Thus, if $\bM$ satisfies
\begin{equation}\label{AnalyticInclusion}
	\exists\, C, h>0 \,\, :\,\, k! \leq Ch^k \bM_k,\quad k\in\Z_+, 
\end{equation}
then ${\mathcal A}(\Omega)\subseteq{\mathcal E}^{\{\bM\}}(\Omega)$.

We finally introduce a condition taken
from \cite{Komatsu1973}:
	\begin{equation}\label{mg}
	\exists A>0:\; \forall j,k\in\Z_+:\quad M_{j+k}\leq A^{j+k}M_jM_k.
\end{equation}
Condition \eqref{mg} implies in particular that ${\mathcal E}^{\{\bM\}}(\Omega)$ is closed under differentiation and that 
$\bM\preceq\mathbf{G}^s$ for some $s>1$, see Matsumoto \cite{10.1215/kjm/1250520659}.

If $\bM$ is a weight sequence then we are also going to use the following sequences
\begin{equation*}
	\qquad \mu_k=\frac{M_k}{M_{k-1}},\qquad 
	\Lambda_k=\sqrt[\leftroot{3}\uproot{1}k]{M_k}.
\end{equation*}

We note that \eqref{LogConvexity} gives that $\Lambda_k\leq \mu_k$ 
for all $k\in\N$ and therefore by \eqref{RootsDivergence} it follows that
\begin{equation}\label{DivisorDivergence}
	\lim_{k\rightarrow\infty} \mu_k=\infty.
\end{equation}
Furthermore, \eqref{AnalyticInclusion} is equivalent to the existence of  $\delta>0$ such that
\begin{equation}\label{AnalyticInclusion2}
	k\leq \delta \Lambda_k.
\end{equation}

From \eqref{mg} it follows that there is $\sigma>0$ such that
\begin{equation}\label{DerivClosed}
\Lambda_{k+1}\leq \sigma\Lambda_k.
\end{equation}
It is easy to see that \eqref{DerivClosed} is the condition (M2') of Komatsu \cite{Komatsu1973} written in terms of the
sequence $(\Lambda_k)_k$. 
This condition  is sufficient to guarantee that $\Rou{\bM}$ is closed under differentiation. 

\begin{Def} \label{AdmissibleDef}
	 A weight sequence satisfying properties \eqref{AnalyticInclusion}  and \eqref{mg}  will be referred to as
an {\sl admissible weight sequence}.
\end{Def}

Clearly the Gevrey sequences $\mathbf{G}^s$ are admissible weight sequences for any $s\geq 1$.

A more general family of admissible weight sequences is defined as follows: Let $s\geq 1$ and $\sigma\geq 0$. 
The weight sequence $\bN^{s,\sigma}$ is given by $N_k^{s,\sigma}=(k!)^s(\log(k+e))^{\sigma k}$.
It is easy to see that $\bN^{s,\sigma}$ is admissible for any choice of $s\geq 1$ and $\sigma\geq 0$. Furthermore $\bN^{s,0}=\mathbf{G}^s$ for all $s\geq 1$ and we have for $s\geq 1$ fixed that
\begin{equation*}
	\mathbf{G}^s\precnapprox\bN^{s,\sigma}\precnapprox\mathbf{G}^{s^\prime}
\end{equation*}
for all $\sigma>0$ and every $s^\prime>s$.

In order to present a weight sequence which is not admissible let $q>1$ be a parameter.
The sequence $\mathbf{L}^q$ given by $L_k^q=q^{k^2}$ is a weight sequence which 
satisfies \eqref{AnalyticInclusion} and \eqref{DerivClosed}. However, since
$\mathbf{G}^s\precnapprox\mathbf{L}^q$ for all $s,q> 1$, we conclude that $\mathbf{L}^q$ cannot satisfy \eqref{mg} for any $q>1$.
\vspp

\begin{Lem}\label{SequenceLemma}\footnote{We thank Prof. Gerhard Schindl for helpful discussions regarding this result.}
	Let $\bM$ be an admissible weight sequence.
	Then there is a constant $\sigma>1$
	such that the following holds: For each $k\in\N$ there is a sequence 
	$(k_j)_j$ such that $\Lambda_{k_0}\leq  \Lambda_k$ and
	\begin{equation*}
		\sigma^{j-1}\Lambda_k\leq \Lambda_{k_j}\leq \sigma^j\Lambda_k,\quad j\in\N.
	\end{equation*}
\end{Lem}
\begin{proof}
	We note that due to \eqref{LogConvexity}
	the sequence $(\Lambda_m)_m$ is increasing and 
	that  $\Lambda_m\rightarrow\infty$ for $m\rightarrow\infty$
	by \eqref{RootsDivergence}.
	Furthermore we can assume that $\sigma>1$ in \eqref{DerivClosed}.
	
	Now fix $k\in\N$. We construct the sequence $k_j$ iteratively.
	First set $k_0=k$. Then $\Lambda_{k_0}=\sigma^0 \Lambda_k$.
	Since $\Lambda_m\rightarrow\infty$ there has to be some $m>k$ such that $\Lambda_k<\Lambda_m$.
	Now let $m_0$ be the smallest integer such that $\Lambda_k<\Lambda_{m_0}$, in particular
	this means \ $\Lambda_k=\Lambda_{m_0-1}$ since the sequence $(\Lambda_m)_m$ is non-decreasing.
	Then we have $\Lambda_{m_0}\leq \sigma \Lambda_{m_0-1}=\sigma \Lambda_k$
	It follows that the set 
	$T_1^k=\{m\in\N:\, \Lambda_k< \Lambda_m\leq \sigma\Lambda_k \}$ is non-empty.
	We choose $k_1$ to be the greatest element of $T^k_1$.
	
	Now assume that we have chosen $k_j$ as the greatest number in 
	$T_j^k=\{m\in\N:\, \sigma^{j-1}\Lambda_k<\Lambda_m\leq \sigma^j \Lambda_k\}$.
	There again has to be some $m>k_j$ such that $\Lambda_{k_j}<\Lambda_{m}$.
	Let $m_j$ be the smallest integer $>k_j$ such that $\Lambda_{k_j}<\Lambda_{m_j}$, 
	which also implies that $\Lambda_{k_j}=\Lambda_{m_{j}-1}$.
	It follows that 
	$\Lambda_{m_j}\leq\sigma \Lambda_{m_j-1}=\sigma \Lambda_{k_j}\leq \sigma^{j+1}\Lambda_k$.
	Thence $T^k_{j+1}=\{m\in\N:\, \sigma^j\Lambda_k<\Lambda_m\leq\sigma^{j+1}\Lambda_k\}$ 
	is non-empty and we choose $k_{j+1}$ to be the greatest element in $T^k_{j+1}$.
\end{proof}

\begin{Def}
	Let $\bM$ be a weight sequence.
	The weight function associated to $\bM$ is given by
	\begin{equation}
		\omega_\bM(t)=\sup_{k\in\Z_+}\log \frac{t^k}{M_k},\qquad t\geq 0.
	\end{equation}
\end{Def}
We recall that the associated weight function $\omega_\bM$ is a continuous and increasing function
on the positive real line, cf.~\cite{zbMATH03075026}.
%We recall that we have the following identity
%\begin{equation*}
%	e^{-\omega_\bM(t)}=\sup
%\end{equation*}

\begin{Lem}\label{FunctionLemma}
	Let $\bM$ be an admissible weight sequence.
	Then the associated weight functions satisfies the following estimate
	\begin{equation}\label{mg2}
		\omega_\bM\left(\Lambda_k\right)\leq Hk,\qquad \fa k\in\N
	\end{equation}
	for some constant $H\geq 1$.
\end{Lem}
\begin{proof}
	Since $\omega_\bM$ is a continuous and increasing function we have that
	$\omega_\bM(\Lambda_k)\leq \omega_\bM(\mu_k)$ for all $k\in\N$.
	Hence it is enough to show that $\omega_\bM(\mu_k)\leq H k$.
	By Mandelbrojt \cite{zbMATH03075026} we know the following fact:
	\begin{equation*}
		\fa k\in\N:\qquad  \omega_\bM\bigl(\mu_k\bigr)=\log \frac{\mu_k^{k}}{M_k}.
	\end{equation*}
	According to Matsumoto \cite{10.1215/kjm/1250520659} condition \eqref{mg} is equivalent to 
	\begin{equation*}
		\exists D>0\;\fa k\in\N:\quad \mu_k\leq D\Lambda_k.
	\end{equation*}
	Here we can choose $D=2A$ where $A$ is the constant from $\eqref{mg}$.
	It follows that 
	\begin{equation*}
		\mu_k^k\leq D^k M_k
	\end{equation*}
	which in turn implies that 
	\begin{equation*}
		\omega_\bM(\mu_k)=\log\frac{\mu_k^k}{M_k}\leq H k
	\end{equation*}
	where $H=\log D$. Since we can assume without loss of generality that $D\geq e$ we have that $H\ge 1$.
\end{proof}

We need also to dwell a little bit on the functional analytic structure of Denjoy-Carleman classes,
for more details see \cite{Komatsu1973}. We shall use the following notation: if $U\subseteq \R^{n}$ is an open subset then $\CB(U)$ denotes
the space of all bounded,  smooth functions on $U$ which have all its derivatives also bounded.
For each weight sequence $\bM$ and all open sets $U\subseteq\R^n$ we can define a norm
on  $\CB(U)$ by
\begin{equation*}
	\|f\|_{\bM,U,h}=\sup_{\alpha\in\Z_+^n}
	\frac{{\|D^\alpha f\|}_{L^\infty(U)}}{\!\!h^{\alp} M_{\alp}}.
\end{equation*}
The resulting Banach space is
\begin{equation*}
	\CB_{\bM,h}(U)=\{f\in \CB(U):\,\|f\|_{\bM,U,h}<\infty\}.
\end{equation*}

Clearly we have $\CB_{\bM,h_1}(U)\subseteq \CB_{\bM,h_2}(U)$ if $h_1\leq h_2$ and thus there is
a continuous and injective map
$\vartheta_{h_1}^{h_2}: \CB_{\bM,h_1}(U) \longrightarrow \CB_{\bM,h_2}(U)$, which is 
 compact if $h_1<h_2$ by \cite[Proposition 2.2]{Komatsu1973}.
We can then introduce the classes of global ultradifferentiable functions on $U$:
%the Roumieu class
\begin{equation*}
	\CB^{\{\bM\}}(U)=\mbox{ind}_{h>0} \CB_{\bM,h}(U)=\mbox{ind}_{\ell\in\N} \CB_{\bM,\ell}(U) \\
%\shortintertext{and the Beurling class}
%	\gBeu{\bM}{U}&=\proj_{h>0}\Ban{\bM}{U}=\proj_{\ell\in\N}\Ban[1/\ell]{\bM}{U}.
\end{equation*}

We observe that $\CB^{\{\bM\}}(U)$ is a (DFS)-space and thus, in particular, a webbed space
\cite[p.~63, (8)]{Koethe1979}. Notice also that
\begin{equation}
\CB^{\{\bM\}}(U)=\{ f\in \CB(U): \exists h>0 \mbox{ such that } \|f\|_{\bM,U,h}<\infty\}.
%\gBeu{\bM}{U}&=\Set{f\in\E(U)\given \fa h>0\;\, \UltraNorm*{f}{U}{\bM}{h}<\infty}.
\end{equation}

Next we localize the preceding concepts. The (local) Denjoy-Carleman class associated to the weight sequence $\bM$ is defined as
\begin{equation*}
	\mathcal{E}^{\{\bM\}}(U)=\mbox{proj}_{V\Subset U} \CB^{\{M\}}(V).
\end{equation*}

As before we have
$$\mathcal{E}^{\{\bM\}}(U) = \{ f\in {\mathcal E}(U) : \forall \,V \Subset U\, \ex h>0 \mbox{ such that }  \|f\|_{\bM,V,h}<\infty\}.$$

It is easy to see that $\CB^{\{\bM\}}(U)$ and $\mathcal{E}^{\{\bM\}}(U)$ are algebras with respect to 
the pointwise operations. We are going to occasionally refer to $f$ to be of class $\{\bM\}$ in $U$
if $f\in {\mathcal E}^{\{\bM\}}(U)$.

\section{The concept of  $\{\bM\}$-hypoellipticity. Statement of the main results.}

\begin{Def}
	Let $U\subseteq\R^n$ be an open set and $\bM$ be a weight sequence.
	If $P=P(x,D)$ is a linear partial differential operator with coefficients in $\Rou{\bM}(U)$ then
	$P$ is $\{\bM\}$-hypoelliptic in $U$ if given $u\in\Dli(U)$ the following holds:
	if $V\subseteq U$ is open and if $Pu\vert_V\in\Rou{\bM}(V)$ then $u\vert_V\in\Rou{\bM}(V)$.
	
	Furthermore $P$ is $\{\bM\}$-hypoelliptic at $x_0\in U$ if there is a neighborhood $U_0$ of $x_0$ such that $P$ is $\{\bM\}$-hypoelliptic in $U_0$.
\end{Def}
M\'{e}tivier \cite{zbMATH03712425}
gave a criterion for analytic (and Gevrey) hypoellipticity at a point in the case
of differential operators $P$ with analytic coefficients in $\Omega$
which satisfy the following condition:
\begin{equation}\label{RightInverse}\tag{\textbf{H}}
	\text{There is a continuous operator }R\negmedspace:\,L^2(\Omega)\rightarrow L^2(\Omega)\,
	\text{ such that }PR=\mathrm{Id}.
\end{equation}
Bove, Mughetti and Tartakoff \cite{BoveMT} extended this characterization to Gevrey hypoellipticity
of  analytic pseudodifferential operators.

Our aim is to generalize M\'etivier's result to $\{\bM\}$-hypoellipticity.
In order to do so we need to introduce a weighted Sobolev norm:
For an open set $U\subseteq\R^n$, a weight sequence $\bM$ and $k\in\N$ we set
\begin{equation*}
\trb u\trb_{U,\bM,k}=\sum_{\alp\leq k} \Lambda_k^{k-\alp}\lVert D^\alpha u\rVert_{L^2(U)}.
\end{equation*}
Our main result is the following theorem:
\begin{Thm}\label{MainThm}
	Let $\bM$ be an admissible weight sequence, 
	$P$ be a differential operator with 
	ultradifferentiable coefficients of class $\{\bM\}$
	in $\Omega$ which satisfies \eqref{RightInverse}.
	
	Then $P$ is $\{\bM\}$-hypoelliptic 
	at  a point $x_0\in \Omega$ if and only if there is a neighborhood $U_0\subseteq\Omega$
	of $x_0$ such that
	for all open sets $V\Subset U\subseteq U_0$ there are constants $C,L>0$ such that
	for all $\Dli(U)$ and all $k\in\Z_+$ we have 
	\begin{gather}
		\label{SmoothCondition1} Pu\in H^k(U)\Longrightarrow u\in H^k(V),\\
		\lVert u\rVert_{H^k(V)}\leq CL^k \left(\trb Pu\trb_{U,\bM,k}
		+M_k\lVert u\rVert_{L^2(U)}\right).
		\label{UltraCondition}
	\end{gather}
\end{Thm} 

	It follows immediately from the condition \eqref{SmoothCondition1} that if a differential
	operator $P$ is $\{\bM\}$-hypoelliptic at $x_0$ for some admissible weight sequence $\bM$  then $P$ is smooth hypoelliptic at $x_0$.
	Furthermore we have the following corollaries.

\begin{Cor}\label{MetivierCor1}
	Let $\bM$ and $\bM^\prime$ be two admissible weight sequences such that $\bM\preceq\bM^\prime$.
	Moreover, let $\Omega\subseteq\R^n$ be an open set and
	$P$ be a differential operator with $\Rou{\bM}(\Omega)$-coefficients of class $\{\bM\}$   such that \eqref{RightInverse} holds.
	If $P$ is $\{\bM\}$-hypoelliptic at $x_0\in\Omega$
	then $P$ is $\{\bM^\prime\}$-hypoelliptic at $x_0$.
\end{Cor}
\begin{proof}[Proof of Corollary \ref{MetivierCor1}]
	If we set $\Lambda_k=(M_k)^{1/k}$ and $\Lambda_k^\prime=(M_k^\prime)^{1/k}$
	then $\bM\preceq\bM^\prime$ implies
	that there are constants $C_1,h$ such that $\Lambda_k\leq \sqrt[k]{C_1}h\Lambda_k^\prime$
	for all $k\in\N$. We can assume that $C_1,h\geq 1$.
	Thus we conclude that
	\begin{equation*}
		\begin{split}
			\trb u\trb_{U,\bM,k}
			&=\sum_{\alp\leq k}\Lambda_k^{k-\alp}\lVert D^\alpha u\rVert_{L^2(U)}\\
			&\leq \sum_{\alp\leq k}C_1^{(k-\alp)/(k)}h^{k-\alp}
			\left(\Lambda_k^\prime\right)^{k-\alp}\lVert D^\alpha u\rVert_{L^2(U)}\\
			&\leq C_1h^k\trb u\trb_{U,\bM^\prime,k}.
		\end{split}
	\end{equation*}
	
	If $P$ is $\{\bM\}$-hypoelliptic at $x_0$ then 
	there is a neighborhood $U_0$ of $x_0$ such that for all $V\Subset U\subseteq U_0$ the condition \eqref{SmoothCondition} holds
	and \eqref{UltraCondition}
	is satisfied for $\bM$ and for some constants $C_2,L$ independent of $k$ and $u\in\Dli(U)$.
	Hence 
	\begin{equation*}
		\begin{split}
			\lVert u\rVert_{H^k(V)}&\leq C_2L^k\left(\trb Pu\trb_{U,\bM,k}+M_k\lVert u\rVert_{L^2(U)}\right)\\
			&\leq C_2L^k\left(C_1h^k\trb Pu\trb_{U,\bM^\prime,k}+C_1h^kM_k^\prime
			\lVert u\rVert_{L^2(U)}\right)\\
			&=C_1C_2(hL)^k\left(\trb Pu\trb_{U,\bM^\prime,k}+M_k^\prime\lVert u\rVert_{L^2(U)}\right)
		\end{split}
	\end{equation*}
	and therefore $P$ is $\{\bM^\prime\}$-hypoelliptic at $x_0$.
\end{proof}
As a special case we obtain
\begin{Cor}\label{AnalyticDistrCor}
	Let $P$ be a differential operator with analytic coefficients in $\Omega\subseteq\R^n$ satisfying \eqref{RightInverse}.
	 If $P$ is analytic hypoelliptic at some point $x_0\in\Omega$ then
	$P$ is $\{\bM\}$-hypoelliptic at $x_0$ for all admissible weight sequences $\bM$.
\end{Cor}

\section{Proof of Theorem \ref{MainThm}}
We start by introducing the Ehrenpreis-H\"ormander cut-off functions: For every open sets 
$V\Subset U\subseteq\R^n$ there exists a sequence $\chi_k\in \mathcal{D}(U)$ with $\chi_k\vert_V\equiv 1$
for all $k\in\N_0$ satisfying 
\begin{equation*}
	\left\lvert D^\alpha \chi_k(x)\right\rvert\leq Q^{\alp} k^{\alp},\qquad \alp\leq k.
\end{equation*}
We will call such a sequence an Ehrenpreis-H\"ormander cut-off sequence which is supported in $U$
and centered in $V$.
Ehrenpreis-H\"ormander cut-off sequences have been heavily used in local and microlocal 
regularity theory in the analytic and ultradifferentiable category.

\begin{Lem}\label{ComparisonLemma}
	Let $(\chi_k)_k$ be an Ehrenpreis-H\"ormander cut-off sequence supported in an open set 
	$U\subseteq\R^n$.
	If $\bM$ is a weight sequence satisfying \eqref{AnalyticInclusion}
	then there is a constant $\gamma>0$ such that
	for all $k\in\N$ and $u\in H^k(U)$ we have
	\begin{equation*}
		\trb\chi_ku\trb_{\R^n,\bM,k}\leq \gamma^{k}\trb u\trb_{U,\bM,k}.
	\end{equation*}
\end{Lem}

\begin{proof}
	Note first if $u\in H^k(U)$ then $\chi_k u$ can be extended to  an element of $H^k(\R^n)$
	by setting $0$ outside $U$. We note also that \eqref{AnalyticInclusion2} gives that
	there is some $\delta>0$ such that $k\leq \delta \Lambda_k$.
	
	The Leibniz rule gives
	\begin{equation*}
		D^\alpha(\chi_k u)=\sum_{\beta\leq \alpha}\binom{\alpha}{\beta}D^\beta\chi_k
		D^{\alpha-\beta}u
	\end{equation*}
	Thus we obtain
	\begin{equation*}
		\begin{split}
			\lVert D^\alpha(\chi_k u)\rVert_{L^2(\R^n)}
			&\leq \sum_{\beta\leq \alpha}\binom{\alpha}{\beta}Q^{\lvert\beta\rvert} k^{\lvert\beta\rvert}
			\lVert D^{\alpha-\beta}u\rVert_{L^2(U)}\\
			&\leq \sum_{\beta\leq \alpha}\binom{\alpha}{\beta}Q^{\lvert\beta\rvert} \delta^{\lvert\beta\rvert} 
			\Lambda_k^{\lvert\beta\rvert} \Lambda_k^{\alp -\lvert\beta\rvert -k}
			\trb u\trb_{U,\bM,k}\\
			&=\sum_{\beta\leq \alpha}\binom{\alpha}{\beta}(Q\delta)^{\lvert\beta\rvert}
			\Lambda_k^{\alp-k}\trb u\trb_{U,\bM,k}.
		\end{split}
	\end{equation*}
	
	It follows that 
	\begin{equation*}
		\begin{split}
			\trb\chi_k u\trb_{\R^n,\bM,k}&= \sum_{\alp\leq k} \Lambda_k^{k-\alp}
			\lVert\chi_ku\rVert_{L^2(\R^n)}\\
			&\leq \left[\sum_{\alp\leq k}\sum_{\beta\leq \alpha}\binom{\alpha}{\beta}
			(Q\delta)^{\lvert\beta\rvert}\right]
			\trb u\trb_{U,\bM,k}
		\end{split}
	\end{equation*}
	and we have proven the Lemma since there is a constant $\gamma>0$ such that
	\begin{equation*}
		\sum_{\alp\leq k}\sum_{\beta\leq \alpha}\binom{\alpha}{\beta}
		(Q\delta)^{\lvert\beta\rvert}\leq \gamma^k.
	\end{equation*}
\end{proof}

\begin{Prop}\label{FourierTransform}
	Let $\bM$ be an admissible weight sequence,
	$\Omega$ be a neighborhood of $x_0$ 
	and let $E$ be a Banach space continuously injected in $L^2(\Omega)$. 
	If we suppose that there is an open set $U_0\subseteq\Omega$ such that
	$u\vert_{U_0}\in\Rou{\bM}(U_0)$ for all $u\in E$ then 
	for any $V\Subset U\Subset U_0$ and all  
	Ehrenpreis-H\"ormander cut-off sequence $(\chi_k)_k$ supported in $U$ and 
	centered on $V$ 
	there exist constants $C,\gamma>0$ such that
	\begin{gather}\label{FourierEq1}
		\lvert\xi\rvert^k\chi_k u\in L^2(\R^n),\\
		\label{FourierEq2}	
		\left\lVert\lvert\xi\rvert^k\widehat{\chi_ku}(\xi)\right\rVert_{L^2(\R^n)}
		\leq C\gamma^k M_k\lVert u\rVert_E
	\end{gather}
	for all $k\in\N$ and $u\in E$.
\end{Prop}

\begin{proof}
We have that the restriction map 
\begin{gather*}
	T_U:E\longrightarrow \CB^{\{\bM\}}(U),\\
	f\longmapsto f\vert_U
\end{gather*}
for $U\Subset U_0$ is a well-defined mapping.
If $(f_j)_j$ is a sequence which converges in $E$ to $0$ and $f_j\vert_U\rightarrow g$ in
$\CB^{\{\bM\}}(U)$ but this implies that $\lVert f_j-g\rVert_{L^2(U)}\rightarrow 0$ since 
$\lVert\,.\,\rVert_{L^2(U)}\leq C\sup_{U}\lvert\,.\,\rvert$ for a constant only depending on $U$.
On the other hand, $\lVert{f_j}\rVert_{L^2(\Omega)}\rightarrow 0$ since $E$ is continuously injected in $L^2(\Omega)$.
It follows that $g=0$ and thus the graph of $T_U$ is closed.
By the version of the closed graph theorem given in \cite[p.56, (1)]{Koethe1979} we have that
$T_U$ is continuous.
If we denote the unit ball in $E$ by $B_1$ then we deduce that $T_U(B_1)$ is bounded and thence
there is
some $h>0$ such that $T_U(B_1)\subseteq \CB_{\bM,h}(U)$ which gives
that $T_U(B)\subseteq \CB_{\bM,h}(U)$. The closed graph theorem for Banach spaces
implies now that $T_U$ is continuous from $E$ to $\CB_{\bM,h}(U)$.
We denote the norm of this map by $C_T$ and obtain for $\alp=k$ that
\begin{equation*}
	\begin{split}
		\left\lVert D^\alpha\bigl(\chi_ku\bigr)\right\rVert_{L^2(\R^n)}
		&\leq \sum_{\beta\leq\alpha}
		C_U\binom{\alpha}{\beta}(Qk)^{\lvert\beta\rvert} h^{\alp-\lvert\beta\vert}
		M_{\alp-\lvert\beta\rvert}
		\lVert u\rVert_{U,\bM,h}\\
		&\leq \sum_{\beta\leq \alpha} \binom{\alpha}{\beta}(Q \delta)^{\lvert\beta\rvert}
		 \Lambda_k^{\lvert\beta\rvert}
		h^{\alp-\lvert\beta\rvert}\Lambda_k^{\alp-\lvert\beta\rvert}\lVert u\rVert_{U,\bM,h}\\
		&\leq \sum_{\beta\leq \alpha} \binom{\alpha}{\beta}(Q\delta)^{\lvert\beta\rvert}
		h^{\alp-\lvert\beta\rvert} M_{\alp}
		C\lVert u\rVert_{E}\\
		&\leq C\left(Q\delta+h\right)^k M_k\lVert u\rVert_{E}.
	\end{split}
\end{equation*}
since $\Lambda_k$ is increasing and  by \eqref{AnalyticInclusion2}
there is some $\delta>0$ such that $k\leq \delta\Lambda_k$.
It follows that
\begin{equation*}
	\left\lVert\lvert\xi\rvert^k\widehat{\chi_ku}\right\rVert^2_{L^2(\R^n)}\leq n^kC^2
	h_1^{2k}M_k^2
	\lVert u\rVert^2_{B}
\end{equation*}
where $h_1=Q\delta+h$.
\end{proof}

If $\bM$ is a weight sequence it is easy to see that 
\begin{equation}\label{FourierNormEstimate}
	\begin{split}
		\int\!\bigl(\Lambda_k+\lvert\xi\rvert\bigr)^{2k}\left\lvert \hat{u}(\xi)\right\rvert^2\,d\xi
		&=\int\left\lvert\bigl(\Lambda_k+\lvert\xi\rvert\bigr)^k\hat{u}(\xi)\right\rvert^2\,d\xi\\
		&\leq 
		\sum_{\ell=0}^k\binom{k}{\ell}^2
		\Lambda_k^{2(k-\ell)}\int\left\lvert\lvert\xi\rvert^\ell \hat{u}(\xi)\right\rvert^2\,d\xi\\
		&\leq 4^k\sum_{\alp\leq k} \Lambda^{2(k-\alp)}\lVert D^\alpha u\rVert^2_{L^2(\R^n)}\\
		&\leq 4^k \bigl(\trb u\trb_{\R^n,\bM,k}\bigr)^2
	\end{split}
\end{equation}
for every $u\in H^k(\R^n)$ and $k\in \Z_+$.
We are also going to use the space
\begin{equation*}
	G_{\{\bM\}}=\left\{u\in L^2(\R^n)\mid e^{\omega_\bM(\lvert\xi\rvert)}\hat{u}\in L^2(\R^n)\right\}.
\end{equation*}
Then $G_{\{\bM\}}\subseteq\Rou{\bM}(\R^n) $ is a Hilbert space with respect to the topology inherited by
$L^2(\R^n)$. 

%Note that $G$ is a Hilbert space of analytic functions with respect to the induced norm.
\begin{Lem}\label{InterpolationLemma1}
	Let $\bM$ be an admissible weight sequence
	and $k\in\N$. Then every $u\in H^k(\R^n)$
	can be written in the form $u=\sum_{j=0}^\infty u_j$ with the $u_j\in G_{\{\bM\}}$ satisfying:
	\begin{equation}\label{InterpolationCondition1}
		\sum_{j=0}^{\infty} \Lambda_{k_j}^{2k}
		\left(\lVert u_j\rVert_{L^2(\R^n)}^2+
		e^{-2\omega_\bM(\Lambda_{k_j})}\lVert u_j\rVert_{G}^2\right)
		\leq C\gamma^k \bigl(\trb u\trb_{\R^n,\bM,k}\bigr)^2
	\end{equation}
	where the constants $C,\gamma>0$ are independent of $k,j$ and $(k_j)_j$ is the sequence from
	Lemma \ref{SequenceLemma}.
\end{Lem}

\begin{proof}
	We may set $k_{-1}=0$ and
	\begin{equation*}
		u_j(x)=(2\pi)^{-n}\negthickspace\int\limits_{\Lambda_{k_{j-1}}
			\leq\lvert\xi\rvert\leq \Lambda_{k_j}}
		e^{ix\xi}\hat{u}(\xi)\,d\xi.
	\end{equation*}
	For $\lvert\xi\rvert\leq \Lambda_{k_j}$ we conclude that
	\begin{equation*}
		\lVert u_j\rVert_{G}^2\leq e^{\omega_\bM(\Lambda_{k_j})}\lVert u_j\rVert_{L^2(\R^n)}.
	\end{equation*}
	
	On the other hand, in the case $\Lambda_{k_{j-1}}\leq \lvert\xi\rvert$ we note first that 
	for $j\geq 2$ we have
	\begin{equation*}
		\Lambda_{k_j}\leq \sigma^j\Lambda_k\leq \sigma^2\Lambda_{k_{j-1}}\leq
		\sigma^2(\Lambda_k+\lvert\xi\rvert).
	\end{equation*}
	Moreover, $\Lambda_{k_1}\leq \sigma\Lambda_k\leq \sigma^2(\Lambda_k+\lvert\xi\vert)$ and
	$\Lambda_{k_0}=\Lambda_k\leq \sigma^2(\Lambda_k+\lvert\xi\rvert)$ since $\sigma\geq1$ and 
	$\lvert\xi\rvert\geq 0$.
	
	Hence, due to  \eqref{FourierNormEstimate},
	\begin{equation*}
		\sum_{j=0}^{\infty}\Lambda_{k_j}^{2k}
		\lVert u_j\rVert_{L^2(\R^n)}^2
		\leq \sigma^{4k}\int\!\left(\lvert\xi\rvert +\Lambda_k\right)^{2k}
		\left\lvert\hat{u}(\xi)\right\rvert^2\,d\xi 
		\leq \gamma^k\left(\trb u\trb_{\R^n,\bM,k}\right)^2
	\end{equation*}
	for some $\gamma>0$.
\end{proof}
\begin{Lem}\label{InterpolationLemma2}
	Assume that $\bM$ is an admissible weight sequence and that
	$E$ is a Banach space which is continuously injected in $L^2(\Omega)$. 
	If there is an open set $U_0\subseteq\Omega$ such that $u\vert_{U_0}\in\Rou{\bM}(U_0)$
	for every $u\in E$
	then for all  $V\Subset U_0$ 
	there exists a constant $C$ such that for all $k\in\N$ and every sequence $u_j\in E$, $j\in\Z_+$,
	satisfying
	\begin{equation}\label{InterpolationCondition2}
		\sum_{j=0}^{\infty}
		\Lambda_{k_j}^{2k}\left(\lVert u_j\rVert_{L^2(\Omega)}^2
		+e^{-2\omega_\bM(\Lambda_{k_j})}\lVert u_j\rVert_{E}^2\right)
		=\Phi^2_k(\mathbf{u})<\infty
	\end{equation}
	the series $u=\sum_j u_j$ converges in $L^2(\Omega)$ and $u\vert_V\in H^k(V)$ with
	\begin{equation*}
		\lVert u\vert_V\rVert_{H^k(V)}\leq C^{k+1}\Phi_k(\mathbf{u}).
	\end{equation*}
\end{Lem}

\begin{proof}
	For $k=1$ the condition \eqref{InterpolationCondition2} implies that 
	$u$ converges absolutely in $L^2(\Omega)$.
	By Proposition \ref{FourierTransform}
	we have that for all open sets $V\Subset U\Subset \Omega$
	and for every Ehrenpreis-H\"ormander cut-off sequence $\chi_k\in\mathcal{D}(U)$ centered in $V$
	there are constants
	$C_0,\gamma>0$ such that for all $k\in\Z_+$ and all $u\in B$:
	\begin{equation}\label{Estimate1}
	\left\lVert\left(\frac{\lvert\xi\rvert}{\gamma\Lambda_k}\right)^k\widehat{\chi_ku}(\xi)
	\right\rVert_{L^2(\R^n)}
		\leq C_0\lVert u\rVert_{B}.
	\end{equation}
	We introduce the following functions:
	\begin{align*}
		\theta(j,\xi)&=e^{-\Lambda_{k_j}}
		\left(\frac{\lvert\xi\rvert}{\gamma\Lambda_{k_j}}\right)^{k_j}\\
		g_j(\xi)&=\left(1+\theta(j,\xi)\right)\widehat{\chi_{k_j}u_j}.
	\end{align*}
	By \eqref{Estimate1} we have that
	\begin{equation*}
		\lVert g_j\rVert_{L^2(\R^n)}\leq \lVert u_j\lVert_{L^2(\Omega)}+C_0e^{-k_j}
		\lVert u_j\lVert_{B}
	\end{equation*}
	which implies that
	\begin{equation}\label{Aim1}
		\begin{split}
			\sum_{j=0}^\infty \Lambda_{k_j}^{2k}\lVert g_j\rVert^2_{L^2(\R^n)}
			&\leq \sum_{j=0}^\infty \Lambda_{k_j}^{2k}
			\left[\lVert u_j\rVert_{L^2(\Omega)}
			+C_0e^{-\omega_\bM(\Lambda_{k_j})}\lVert u_j\rVert_{B}\right]^2\\
			&\leq \bigl(1+C_0^2\bigr)\Phi_k(\mathbf{u})^2<\infty.
		\end{split}
	\end{equation}
	
	It is clear that the function $v=\sum_{j=0}^\infty \chi_{k_j}u_j$ coincides with $u$ on $V$.
	It suffices to show that $v\in H^k(\R^n)$ satisfies the estimate:
	\begin{equation*}
		\lVert v\rVert_{H^k(\R^n)}\leq C^{k+1}\Phi_k(\mathbf{u}).
	\end{equation*}
	We have that
	\begin{equation}
		\lVert v\rVert_{L^2(\R^n)}\leq \sum_{j=0}^\infty
		\lVert u_j\rVert_{L^2(\Omega)}\leq 2\Phi_k(\mathbf{u})
	\end{equation}
	and thus  it will be enough to show that $\lvert\xi\rvert^k \hat{v}\in L^2(\R^n)$ and
	\begin{equation}\label{Aim2}
		\left\lVert\lvert\xi\rvert^k\hat{v}(\xi)\right\rvert_{L^2(\R^n)}
		\leq C^{k+1}\Phi_k(\mathbf{u})
	\end{equation}
	holds, where $C>0$ is independent of $k$. We write
	\begin{equation*}
		\lvert\xi\rvert^k\hat{v}(\xi)
		=\sum_{j=0}^\infty\bigl(1+\theta(j,\xi)\bigr)^{-1}g_j(\xi)\lvert\xi\rvert^k
	\end{equation*}
	and conclude that
	\begin{equation*}
		\lvert\xi\rvert^{2k}\left\lvert\hat{v}(\xi)\right\rvert^2\leq 
		\left(\sum_{j=0}^\infty \lvert g_j(\xi)\rvert^2\Lambda_{k_j}^{2k}\right)
		\Theta(\xi)
	\end{equation*}
	where 
	\begin{equation*}
		\Theta(\xi)=\sum_{j=0}^\infty \Lambda_{k_j}^{2k}
		\bigl(1+\theta_j(j,\xi)\bigr)^{-2}.
	\end{equation*}
	Hence \eqref{Aim2} (and thus the Lemma) is proven by  \eqref{Aim1} and the estimate
	\begin{equation}\label{Aim3}
		\lVert\Theta(\xi)\rVert_{L^\infty(\R^n)}\leq C^{k+1}
	\end{equation}
	where $C$ is some constant  independent of $k$.
	In order to establish \eqref{Aim3} we set 
	\begin{equation*}
		\Psi_j(\xi)=\left(\frac{\lvert\xi\rvert}{\Lambda_{k_j}}\right)^{2k}
		\bigl(1+\theta(j,\xi)\bigr)^{-2}
	\end{equation*}
	If $e^{2H} \gamma \Lambda_{k_j}\leq \lvert\xi\rvert$, where $H$ is the constant from \eqref{mg2},
	then we have that
	\begin{equation*}
		\begin{split}
			\Psi_j(\xi)&\leq 
			\left(\frac{\lvert\xi\rvert}{\Lambda_{k_j}}\right)^{2k}e^{2\omega_{\bM}(\Lambda_{k_j})}\gamma^{k_j}
			\left(\frac{\lvert\xi\rvert}{\Lambda_{k_j}}\right)^{-2k_j}
			\\
			&\leq \gamma^{2k}\left(\frac{\lvert\xi\rvert}{\gamma\Lambda_{k_j}}\right)^{2(k-k_j)}
			e^{2\omega_\bM(\Lambda_{k_j})}\\
			&\leq (\gamma e^{2H})^{2k}\exp\left(2\omega_{\bM}(\Lambda_{k_j})-4Hk_j\right)\\
			&\leq (\gamma e^{2H})^{2k} \exp\left(2Hk_j-4Hk_j\right)\\
			&\leq (\gamma e^{2H})^{2k}e^{-2H k_j}
		\end{split}
	\end{equation*}
	by Lemma \ref{FunctionLemma}.
	This gives 
	\begin{equation*}
		\sum_{e^2 \gamma\Lambda_{k_j}\leq \lvert\xi\rvert}
		\Psi(\xi)\leq \left(\gamma e^{2H}\right)^{2k}
		\sum_{j=0}^\infty e^{-2Hj}\leq C_1 \left(\gamma e^{2H}\right)^{2k}.
	\end{equation*}

	On the other hand, if $e^{2H}\gamma \Lambda_{k_j}\geq \lvert\xi\rvert$ 
	then we  estimate 
	\begin{equation*}
		\Psi_j(\xi)\leq \left(\frac{\lvert\xi\rvert}{\Lambda_{k_j}}\right)^{2k}.
	\end{equation*}
	If we set $j_0=\min\{j\in\Z_+\mid \gamma e^{2H}\Lambda_{k_j}\geq \lvert\xi\rvert\}$ 
	for a fixed $\xi$ then we have that
	\begin{equation*}
		\begin{split}
			\Psi(\xi)&\leq \left(\frac{\lvert\xi\rvert}{ \Lambda_{k_j}}\right)^{2k}
			\leq \left(\frac{\lvert\xi\rvert}{\Lambda_{k_{j_0}}}\right)^{2k}
			\left(\frac{\Lambda_{k_{j_0}}}{\Lambda_{k_j}}\right)^{2k}\\
			&\leq \left(\gamma e^{2H}\right)^{2k}\sigma^{j_0-j+1}
		\end{split}
	\end{equation*}
	and conclude that
	\begin{equation*}
		\sum_{\lvert\xi\rvert\leq e^{2H}\gamma\Lambda_{k_j}}\Psi_j(\xi)
		\leq \left(\gamma e^{2H}\right)^{2k}\sigma \sum_{j=0}^\infty \sigma^{-j}
		\leq C_2\left(\gamma e^{2H}\right)^{2k}
	\end{equation*}
\end{proof}
If $P$ is a differential operator with smooth coefficients then we set
\begin{equation*}
	\mathcal{P}_0(U)=\{u\in L^2(U)\mid Pu=0\}
\end{equation*}
for an open set $U$. Clearly $\mathcal{P}_0 (U)$ is a closed subspace of $L^2(U)$.
\begin{Prop}\label{HomogProp}
	Let $\Omega\subseteq\R^n$ be open, $\bM$ be an admissible weight sequence. 
	Furthermore assume that $P$ is a linear
	differential operator with $\Rou{\bM}(\Omega)$-coefficients.
	
	If $P$ is $\{\bM\}$-hypoelliptic at some point $x_0\in\Omega$ then there is a neighborhood $U_0$ of $x_0$ such that   
	for all open sets  $V\Subset U\subseteq U_0$ there are constants $C,h>0$
	such that for all $k\in\Z_+$ and all $u\in L^2(U)$ with $Pu=0$
	we have
	\begin{equation}\label{HomogeneousEstimate}
		\lVert u\rVert_{H^k(V)}\leq C h^k M_k\lVert u\rVert_{L^2(U)}.
	\end{equation}
\end{Prop}

\begin{proof}
	If $V\Subset U$ are open sets and $h>0$ then we set
	\begin{equation*}
		\lVert u\rVert^\bullet_{V,\bM,h}=\sup_{k\in\N_0}\frac{\lVert u\rVert_{H^k(V)}}{h^kM_k}
	\end{equation*}
	for $u\in\E(U)$.
	We define
	\begin{equation*}
		H_{\bM,h}(V)=\{u\in \E(U)\mid \lVert u\rVert_{V,\bM,h}<\infty\}
	\end{equation*}
	and introduce
	\begin{equation*}
		H^{\{\bM\}}(V)=\mbox{ind}_{h>0}H_{\bM,h}(V).%\\
	\end{equation*}
	We note that if $u\in \CB_{\bM,h}(V)$ then 
	\begin{equation*}
		\begin{split}
			\lVert u\rVert_{H^k(V)}^2&=\sum_{\alp\leq k} \lVert D^\alpha u\rVert_{L^2(V)}^2\\
			&\leq C^2\sum_{\alp\leq k} h^{2\alp}M^2_{\alp}\\
			&\leq C^2 M^2_k\sum_{j=0}^k h^{2k}\sum_{\alp=j}1\\
			&\leq C^2 M^2_k\sum_{j=0}^k \binom{n+j-1}{j}h^{2j}\\
			&\leq C^2 M^2_k\sum_{j=0}^k \binom{k}{j}h^{2j} \frac{(n+j-1)!(k-j)!}{(n-1)!k!}\\
			&\leq C^2 M^2_k\binom{n+k-1}{k}\left(1+h^2\right)^k\\
			&\leq C^2 M^2_k 2^{n-1}2^k\left(1+h^2\right)^k.
		\end{split}
	\end{equation*}
	Hence there is a constant $A$ such that
	\begin{equation*}
		\lVert u\rVert_{V,\bM,h_1}^\bullet\leq A\lVert u\rVert_{V,\bM,h}
	\end{equation*}
	where $h_1=2(1+h^2)$.
	Thus we obtain that there is a continuous embedding
	\begin{align*}
		\iota_V:\;\CB_{\bM,h}(V)&\longrightarrow H_{\bM,h_1}(V)\\
		\intertext{and consequently}
		\CB^{\{\bM\}}(V)&\longrightarrow H^{\{\bM\}}(V).
	\end{align*}
	
	Since $P$ is $\{\bM\}$-hypoelliptic at $x_0$, we know that there is a neighborhood
	$U_0\subseteq\Omega$ of $x_0$ such that  
	$u\vert_V\in \CB^{\{\bM\}}(V)$ for all $u\in \mathcal{P}_0(U)$ and all pairs $V\Subset U\subseteq U_0$.
	Then similarly to the proof of Proposition \ref{FourierTransform} we observe that
	the graph of the restriction map 
	\begin{equation*}
		T_V:\;\mathcal{P}_0(U)\longrightarrow \CB^{\{\bM\}}(V)
	\end{equation*}
	is closed. Hence the Closed Graph Theorem of De Wilde implies that $T_V$ is continuous
	and therefore the map
	\begin{equation*}
		S_V=\iota_V\circ T_V:\, \mathcal{P}_0(U)\longrightarrow H^{\{\bM\}}(V)
	\end{equation*}
	is continuous. Again as before in the proof of Proposition \ref{FourierTransform}
	we thus can conclude that for all $V\Subset U$ there exists
	$h>0$ such that the map $u\mapsto u\vert_V$ is continuous from $\mathcal{P}_0(U)$ to $H^{\{\bM\}}(V)$
	which is equivalent to the existence of some constant $C>0$ such that
	\begin{equation*}
		\lVert u\rVert^\bullet_{V,\bM,h}\leq C\lVert u\rVert_{L^2(U)}\qquad 
		\fa u\in\mathcal{P}_0(U).
	\end{equation*}
\end{proof}

\begin{proof}[Proof of Theorem \ref{MainThm}]
	If we assume that $u\in\Dli(U)$, $U$ being an open subset of $U_0$, is such that 
	$Pu\in\Rou{\bM}(U)$ then (1) implies that $u\vert_V\in \E(V)$ 
	for any $V\Subset U$. In particular, $u\in H^k(V)$ for all $k\in\N_0$.
	We observe also that we have for a weight sequence $\bM$ that
	\begin{equation*}
		\begin{split}
			\trb Pu\trb_{U,\bM,k}&=\sum_{\alp\leq k} \Lambda_k^{k-\alp}
			\lVert D^\alpha (Pu)\rVert_{L^2(U)}\\
			&\leq C_0 \sum_{\alp\leq k} \Lambda_k^{k-\alp}h_0^{\alp} M_\alp\\
			&\leq C_0 \sum_{\alp\leq k}\Lambda^{k-\alp}
			\Lambda_k^{\alp}h_0^k\\
			&\leq C_0 M_k\sum_{\alp\leq k} h_0^{\alp}\\
			&\leq C_1 h_1^k M_k
		\end{split}
	\end{equation*}
	for some constants $C_1,h_1>0$ independent of $k$. Thence \eqref{UltraCondition} implies that 
	\begin{equation*}
		\begin{split}
			\lVert u\rVert_{H^k(V)}
			&\leq CL^k\left[\trb Pu\trb_{U,\bM,k}+M_k\lVert u\rVert_{L^2(U)}\right]\\
			&\leq CL^k\left[C_1h_1^k M_k+C^\prime M_k\right]\\
			&\leq C_2h_2^k M_k
		\end{split}
	\end{equation*}
	for some $C_2,h_2>0$.
	Since $\lVert D^\alpha u\rVert_{L^2(V)}\leq\lVert u\rVert_{H^{\alp}(V)}$ 
	it follows that $u$ is ultradifferentiable of class $\{\bM\}$ in $V$. 
	Since $V\Subset U$ is arbitrary we have actually that $u\in\Rou{\bM}(V)$.
	
	On the other hand assume now that $P$ is a differential operator which is 
	$\{\bM\}$-hypoelliptic in $\Omega$ and satisfies \eqref{RightInverse}, i.e.\ there is 
	a continuous map $R:L^2(\Omega)\rightarrow L^2(\Omega)$ such that $PR=\mathrm{Id}$.

	If 
	\begin{equation*}
		G_{\{\bM\}}=\{u\in L^2(\R^n)\mid e^{\omega_\bM(\lvert\xi\rvert)}\hat{u}\in L^2(\R^n)\}
	\end{equation*} 
	is the space from Lemma \ref{InterpolationLemma1} 
	then we set $\tilde{G}_{\{\bM\}}=\{f\vert_\Omega\mid f\in G_{\{\bM\}}\}$.
	The map 
	\begin{equation*}
		G_{\{\bM\}}\ni f\longmapsto R(f\vert_\Omega) \in L^2(\Omega)
	\end{equation*}
	is injective and 
	\begin{equation*}
		E=\{R(f)\mid f\in \tilde{G}_{\{\bM\}}\}
	\end{equation*} 
	is a Banach space with the norm inherited from $G_{\{\bM\}}$.
	It is clear that the injection $E\rightarrow L^2(\Omega)$ is continuous and
	since $P$ is $\{\bM\}$-hypoelliptic in $\Omega$ we can infer the following fact:
	For all $u\in E$ we have that $Pu\in\tilde{G}$ is  of class $\{\bM\}$
	in $\Omega$, but this implies $u$ is of class $\{\bM\}$ in $\Omega$.
	Hence we can use Lemma \ref{InterpolationLemma2}. 
	
	Let $C_0$ be the constant appearing in Lemma \ref{InterpolationLemma2}
	and $V\Subset U\subseteq \Omega$ be arbitrary open sets.
	We choose an Ehrenpreis-H\"ormander cut-off sequence $\chi_j\in\mathcal{D}(U)$ centered in $V$.
	Let $u\in \Dli(U)$ be such that $Pu\in H^k(U)$. We set $f=\chi_k Pu$. 
	By Lemma \ref{ComparisonLemma} we have that
	\begin{equation}\label{MainEstimate1}
		\trb f \trb_{\R^n,\bM,k}\leq \gamma^k\trb Pu \trb_{U,\bM,k}
	\end{equation}
	for some constant $\gamma$ independent of $k$.
	According to Lemma \ref{InterpolationLemma1} we can write 
	\begin{equation*}
		f=\sum_{j=0}^\infty f_j
	\end{equation*}
	with $f_j\in G_{\{\bM\}}$ and 
	\begin{equation*}
		\sum_{j=0}^{\infty}\Lambda_{k_j}^{2k}\left(\lVert f_j\rVert^2_{L^2(\R^n)}
		+e^{-2\omega_{\bM}(\Lambda_{k_j})}\lVert f_j\rVert^2_{G}\right)
		\leq 2\gamma_1^k\bigl(\trb f\trb_{\R^n,\bM,k}\bigr)^2.
	\end{equation*}
	Setting $v_j=R(f_j\vert_\Omega)\in E$ we obtain that
	\begin{equation*}
		\sum_{j=0}^{\infty}\Lambda_{k_j}^{2k}\left(\lVert v_j\rVert^2_{L^2(\Omega)}
		+e^{-2\omega_{\bM}(\Lambda_{k_j})}\lVert v_j\rVert^2_{B}\right)
		\leq 2(1+\lVert R\rVert^2)\gamma_1^k\bigl(\trb{f}\trb_{\R^n,\bM,k}\bigr)^2
	\end{equation*}
	where $\lVert R\rVert$ is the norm of $R$ in $\mathcal{L}(L^2(\Omega))$.
	The series $v=\sum_{j=0}^{\infty}v_j$ converges in $L^2(\Omega)$ by Lemma 
	\ref{InterpolationLemma2} and furthermore we have that $v\vert_V\in H^k(V)$ and
	\begin{equation}\label{MainEstimate2}
		\lVert v\vert_V\rVert_{H^k(V)}\leq C_1^{k+1}\trb f\trb_{\R^n,\bM,k}
	\end{equation}
	for some constant $C_1$ independent of $k$.
	
	We obtain actually that
	\begin{equation*}
		P\left((u-v)\vert_V\right)=0.
	\end{equation*}
	Since $P$ is $\{\bM\}$-hypoelliptic in $U_0$
	it follows that  $u-v\in\Rou{\bM}(V)$. 
	Applying Proposition \ref{HomogProp}  we infer that for each open set $W\Subset V$
	there is a constant $C_2>0$ independent of $u$ (and $v$) such that
	\begin{equation}\label{MainEstimate3}
		\fa k\in\Z_+:\quad \lVert{(u-v)\vert_W}\rVert_{H^k(W)}\leq C_2^{k+1}M_k
		\lVert(u-v)\vert_V\rVert_{L^2(V)}.
	\end{equation}
	We note also that $v=R(f\vert_\Omega)$ and therefore
	\begin{equation}\label{AuxEstimate1}
		\lVert v\rVert_{L^2(\Omega)}\leq \lVert R\rVert
		\lVert f\rVert_{L^2(\R^n)}\leq \lVert R\rVert \lVert Pu\rVert_{L^2(U)}.
	\end{equation}
	Furthermore it is easy to see that
	\begin{equation}\label{AuxEstimate2}
		M_k\lVert Pu\rVert_{L^2(U)}\leq \trb Pu\trb_{U,\bM,k}.
	\end{equation}
	We are now able to finish the proof:
	\begin{equation*}
		\begin{split}
			\lVert u\rVert_{H^k(W)}&\leq \lVert v\rVert_{H^k(W)}+\lVert u-v\rVert_{H^k(W)}\\
			&\leq C_1^{k+1}\trb f\trb_{\R^n,\bM,k}+C_2^{k+1}M_k\lVert u-v\rVert_{L^2(V)}\\
			&\leq (C_1\gamma)^{k+1}\trb Pu\trb_{U,\bM,k}+C_2^{k+1}M_k
			\left[\lVert u\rVert_{L^2(V)}+\lVert v\rVert_{L^2(V)}\right]
		\end{split}
	\end{equation*}
	by \eqref{MainEstimate1}, \eqref{MainEstimate2} and \eqref{MainEstimate3}.
	Applying \eqref{AuxEstimate1} and \eqref{AuxEstimate2} we see that
	\begin{equation*}
		M_k\lVert v\rVert_{L^2(V)}\leq \lVert R\rVert\,\trb Pu\trb_{U,\bM,k}.
	\end{equation*}
	Hence we have shown that there are constants $C,h>0$ such that for all
	$k\in\Z_+$ and every $u\in\Dli(U)$ with $Pu\in H^k(U)$ the following estimate holds:
	\begin{equation*}
		\lVert u\rVert_{H^k(V)}\leq C h^k\left[\trb Pu \trb_{U,\bM,k}+ M_k\lVert u\rVert_{L^2(U)}\right].
	\end{equation*}
\end{proof}

\section{The hyperfunction case - Final remarks}

Denote by $\mathrm{B}$ the sheaf of (germs of) hyperfunctions in $\R^n$. 
We strength the definition of hypoellipticity in the following way: 

\begin{Def}
	Let $U\subseteq\R^n$ be an open set and $\bM$ be a weight sequence.
	If $P=P(x,D)$ is a linear differential operator with real-analytic coefficients then
	$P$ is $\{\bM\}$-hypoelliptic in $U$ in the hyperfunction sense  if given $u\in\mathrm{B}(U)$ the following holds:
	if $V\subseteq U$ is open and if $Pu\vert_V\in\Rou{\bM}(V)$ then $u\vert_V\in\Rou{\bM}(V)$.
	\end{Def}
	
	Furthermore $P$ is $\{\bM\}$-hypoelliptic at $x_0\in U$ in the hyperfunction sense if there is a neighborhood $U_0$ of $x_0$ such that $P$ is $\{\bM\}$-hypoelliptic in $U_0$ in the hyperfuntion sense.
\vsp

A close look at the proof of Theorem 1 leads to the following result:
\begin{Thm}\label{MainThm2}
	Let $\bM$ be an admissible weight sequence and
	$P$ be a differential operator with 
	real analytic coefficients 
	in $\Omega$ satisfying \eqref{RightInverse}.
	
	Then $P$ is $\{\bM\}$-hypoelliptic 
	at  a point $x_0\in \Omega$ in the hyperfunction sense if and only if there is a neighborhood $U_0\subseteq\Omega$
	of $x_0$ such that
	for all open sets $V\Subset U\subseteq U_0$ there are constants $C,L>0$ such that
	for all $\mathrm{B}(U)$ and all $k\in\Z_+$ we have 
	\begin{gather}
		\label{SmoothCondition} Pu\in H^k(U)\Longrightarrow u\in H^k(V),\\
		\lVert u\rVert_{H^k(V)}\leq CL^k \left(\trb Pu\trb_{U,\bM,k}
		+M_k\lVert u\rVert_{L^2(U)}\right).
		\label{UltraCondition}
	\end{gather}
\end{Thm} 

An operator $P=P(x,D)$  in an open set $W\subseteq\R^n$ will be said to belong to the H\"ormander class $\mathfrak{H}(W)$ if it can be written in the form
$$             P = \sum_{j=1}^\nu X_j^2 , $$
where each $X_j$ is a real-valued, real analytic vector field defined in $W$ and the following condition is satisfied:
the Lie algebra spanned by $X_1,\ldots,X_\nu$ has rank equal to $n$ at any point of $W$. 

It can be proved that every point in $W$ has an open neighborhood $\Omega$ such that if $P\in\mathfrak{H}(W)$
then  $P$  (and also its transpose!) satisfy \eqref{RightInverse}. 
In particular, Theorem \ref{MainThm2} applies to $P\in\mathfrak{H}(W)$ restricted to $\Omega$. 
It is also well known  (cf.~\cite{DerridjZuily}) that given
$P\in \mathfrak{H}(W)$ and $x_0\in W$ then there is $s_0\geq 1$ so that $P$ is $\left\{\mathbf{G}^{s_0}\right\}$-hypoelliptic at $x_0$. 
According to Cordaro-Hanges \cite{CordaroHanges}
it follows that  $P$ is $\left\{\mathbf{G}^{s_0}\right\}$-hypoelliptic at $x_0$ in the hyperfunction sense.  By an elementary extension of Corollary \ref{MetivierCor1} to this hyperfunction set up 
it follows that $P$ is then $\{\bM\}$-hypoelliptic at $x_0$ in the hyperfunction sense  for every weight sequence $\mathbf{G}^{s_0}\preceq \bM$.

\section*{Acknowledgments}
\n The first author was partially supported by Conselho Nacional de Desenvolvimento Cient\'ifico e Tecnol\'ogico-CNPq, Grant 303945/2018-4 and S\~ao Paulo Research Foundation (FAPESP), Grant 2018/14316-3. The second author supported by the Austrian Science Fund (FWF) Grant J4439 in order to develop a pos doctoral program at the University of São Paulo, São Paulo, Brazil, under the supervision of the first author.

\bibliographystyle{abbrv}
%\bibliography{Metivier}

\vskip 0.3in

\begin{minipage}[b]{8 cm}
	{\bf Paulo D. Cordaro} 
	\\ 
	University of S\~ao Paulo\\ 
	S\~ao Paulo, Brazil\\
	E-mail: {\it cordaro@ime.usp.br}
\end{minipage}
\hfill
\begin{minipage}[b]{8 cm}
	{\bf Stefan Fürdös} ({\it \small Corresponding author})\\
	University of S\~ao Paulo  \\
	S\~ao Paulo, Brazil\\
	E-mail: {\it stefan.fuerdoes@univie.ac.at}
\end{minipage}

\end{document}